\documentclass[11pt]{preprint}
\usepackage{amsmath,amsfonts,latexsym,amssymb,amscd,
  graphicx,enumerate,amsthm}
\usepackage{ulem}
\usepackage{cancel}
\usepackage{hyperref}

\usepackage[affil-it]{authblk}

\usepackage[usenames]{color}

\allowdisplaybreaks[1]

\newcommand{\bpf}[1][Proof]{{\noindent {\sc #1: }}}
\newcommand{\epf}{{{\hfill $\Box$ \smallskip}}}
\newcommand{\ONE}{{\mathbf{1}}}

\newcommand{\Pp}{\mathsf{P}}

\newcommand{\E}{\mathsf{E}}

\newcommand{\R}{{\mathbb R}}

\def\cI{\mathcal{I}}
\def\cF{\mathcal{F}}
\def\cC{\mathcal{C}}
\def\cV{\mathcal{V}}

\newcommand{\tauy}{\bar\tau}
\newcommand{\ttau}{\bar{\bar\tau}}

\newcommand{\eps}{{\varepsilon}}
\newcommand{\e}{\varepsilon}

\iftrue
\else
  \RequirePackage{geometry}
\fi

\newtheorem{theorem}{Theorem}[section]
\newtheorem{remark}{Remark}[section]
\newtheorem{lemma}{Lemma}[section]

\newtheorem{proposition}{Proposition}[section]

\numberwithin{equation}{section}

\renewenvironment{proof}[1][Proof]{
{\noindent {\sc #1: }}
}{
{{\hfill $\Box$ \smallskip}}
}
\makeatletter
\let\orgdescriptionlabel\descriptionlabel
\renewcommand*{\descriptionlabel}[1]{%
  \let\orglabel\label
  \let\label\@gobble
  \phantomsection
  \edef\@currentlabel{#1}%
  \let\label\orglabel
  \orgdescriptionlabel{#1}%
}
\makeatother

\begin{document}

\title{Tails of exit times from unstable equilibria on the line}
\author{Yuri Bakhtin, Zsolt Pajor-Gyulai}
\affil{Courant Institute of Mathematical Sciences\\
New York University\\
New York, NY, USA\\
}
\maketitle
\begin{abstract}
For a one-dimensional smooth vector field in a neighborhood of an unstable equilibrium, we consider the associated
dynamics perturbed by small noise. We give a revealing elementary proof of a result proved earlier using heavy machinery from Malliavin calculus. In particular, we obtain precise vanishing noise asymptotics for the tail of the exit time and for the exit distribution conditioned on atypically long exits.
\end{abstract}

\section{Introduction}
In a recent paper~\cite{BPG2017}, we studied tails of diffusion exit times 
from neighborhoods of unstable critical points, in the limit of vanishing noise, in one dimension. 
The typical exit time $\kappa_\eps$ in this setting (see the detailed description of the setting below) is of the order of $\frac{1}{\lambda}\log\frac{1}{\eps}$, where $\lambda>0$ is the local expansion coefficient of the linearization of the system near the critical point,  and $\eps\downarrow 0$ is the noise magnitude. 
The main result of~\cite{BPG2017} is that the following polynomial asymptotics holds for a class of initial conditions near the critical point:
\begin{equation}
\label{eq:recalling-old-result}
\Pp\left(\kappa_\eps>\frac{\alpha}{\lambda}\log \frac{1}{\eps}\right)=c\eps^{\alpha-1}(1+o(1)),\quad \eps\downarrow 0,
\end{equation}
for $\alpha>1$, with an explicit dependence of the factor $c$ on the initial condition and the parameters of the model, see \eqref{eq:main-thm-claim-1} below for details.

This result is a part of an ongoing effort to understand the long-term properties of multi-dimensional diffusions in the context of noisy heteroclinic networks, including the limiting behavior of invariant distributions associated with such systems in the vanishing noise limit. 
The typical behavior in such settings is understood for time scales logarithmic in~$\eps^{-1}$, see~\cite{Bak2010}, \cite{Bak2011}, \cite{AB2011}. To see what happens in the long run though, one has to quantify rare events responsible for
transitions that are atypical at the logarithmic time scale. Our work in progress shows that
these rare events play a crucial role in the long-term dynamics near noisy heteroclinic networks. 
Moreover, we argue that they occur exactly due to atypically long stays near unstable critical points. The resulting picture is similar to that of metastability but with polynomial transition rates in place of exponential ones.  We give more details on this picture in Section~\ref{sec:heteroclinic} while here we only reiterate that the result of the form~\eqref{eq:recalling-old-result} and its ramifications will be crucial for that program. However,
the technique we used in~\cite{BPG2017} to analyze densities of auxiliary random variables, was based on heavy tools from Malliavin calculus. That approach 
somewhat obscures the reason why this result is true and does not seem to be tractable when applied to the study of the analogous exit problem in the neighborhood of a hyperbolic saddle in $\R^d$, $d>1$, i.e. when both attracting and repelling directions are present.

In the present note, our goal is to give a new proof of this result that  (a) is based on a more precise description of the dynamics at small scales, (b) uses more elementary tools of stochastic calculus, and (c) has a strong potential to be applicable in higher dimensions. In fact, we prove a slightly more general result on probabilities of the form $\Pp\left(\kappa_\eps>\frac{\alpha}{\lambda}\log \frac{1}{\eps}+t\right)$, $\alpha>1$, for all $t\in\R$ instead of $t=0$ considered in~\cite{BPG2017}. It turns out that,
asymptotically, the dependence on $t$ is exponential, which implies that for any $T\in\R$,
$\kappa_\eps-\frac{\alpha}{\lambda}\log \frac{1}{\eps}-T$ conditioned on $\kappa_\eps-\frac{\alpha}{\lambda}\log \frac{1}{\eps}>T$
converges in distribution to an exponential random variable. This phenomenon is a manifestation of loss of memory in the system under conditioning and it is consistent with the fact that 
$\kappa_\eps - \frac{1}{\lambda}\log \frac{1}{\eps}$ converges in distribution to a random variable with exponentially decaying right tails, see e.g. \cite{Bak2010}.

An important ingredient in this note is a conditional equidistribution result (Lemma~\ref{lem:X-equid}) that states that the distribution of the diffusion, conditioned on no exit from a small interval, converges 
to the uniform distribution.  Thus our new approach is closer in the spirit to the one based on quasi-stationary distributions, see~\cite{Champagnat2016}. However, the existing general theory does not provide
answers for us since in our situation both the system and the domain depend on~$\eps$. Moreover, the time scales we are interested in are too short for the $t\to\infty$ limit to be a good approximation while taking $\e\downarrow 0$.

\bigskip

Let us be more precise now. We consider the family of stochastic differential equations
\begin{equation}\label{eq:SDE}
dX_{\eps}(t)=b\left(X_{\eps}(t)\right)dt+\eps\sigma\left(X_{\eps}(t)\right)dW(t),
\end{equation}
on a bounded interval $\cI=[q_-,q_+]\subseteq \mathbb{R}$ with origin in its interior. The drift is given by a  vector field~$b\in\mathcal{C}^{2}(\R)$ and the random perturbation is given via a standard Brownian motion $W$
with respect to a filtration $(\cF_t)_{t\ge 0}$ defined on some probability space $(\Omega,\mathcal{F},\Pp)$ under the usual conditions. The noise magnitude is given by a small parameter $\eps>0$ in front of the diffusion coefficient $\sigma$, which is assumed to be Lipschitz and satisfy $\sigma(0)>0$.
Although we are interested only
in the evolution within~$\cI$, we
can assume that $b$ and $\sigma$ are globally Lipschitz without changing the setting.

Standard results on stochastic differential equations (see, e.g.,~\cite[Chapter 5]{KS1991}) imply that for any starting location $X^{\eps}(0)\in\cI$, the equation~\eqref{eq:SDE} has a unique strong solution up to
the exit time from $\cI$,
\[
\tau_{\cI}^{\eps}=\inf\{t\geq 0: X_{\eps}(t)\in\partial \cI\}.
\]

Let $(S^t)_{t\in\R}$ be the flow generated by the vector field $b$, i.e., $x(t)=S^tx_0$ is the solution of the autonomous ordinary differential equation
\begin{equation*}
\dot{x}(t)=b(x(t)),\qquad x(0)=x_0\in\R
\end{equation*}
(we recall that solving this equation for negative times is equivalent to solving $\dot y(t)=-b(y(t))$ for $y(t)=x(-t)$).
We assume that there is a unique repelling zero of the vector field $b$ on $\R$, which, without loss of generality, we
place at the origin. In other words, we assume that $b(0)=0$ and, for some $\lambda>0$ and $\eta\in\cC^2(\cI)$,
\begin{equation}\label{eq:linearizable}
b(x)=\lambda x+\eta(x)|x|^2,\qquad x\in\cI.
\end{equation}
Note that since the origin is the only zero of $b$ in the closed interval $\cI$, this assumption implies that for all $x\neq 0$, there is a uniquely defined
finite time $T(x)$ such that $S^{T(x)}\in\partial\cI$.

Under
the condition \eqref{eq:linearizable},  the map $f:\cI\to\R$ defined by
\begin{equation}\label{eq:coord-change-diffeo}
f(x)=\lim_{t\to\infty}e^{\lambda t} S^{-t}x=x-\int_0^\infty e^{\lambda s}\eta(S^{-s}x)|S^{-s}x|^2ds
\end{equation}
is an order preserving $\cC^2$-diffeomorphism
(see \cite{Eizenberg:MR749377}). In particular, $f(q_-)<0<f(q_+)$. This map linearizes the flow $(S^t)$ (see~\eqref{eq:conjugation})  and helps to state the main result concisely, see~\eqref{eq:main-thm-claim-1}.

Under the above assumptions, a version of the following theorem was proved in \cite{BPG2017}. In its statement and throughout the paper we use
\begin{equation}
\label{eq:Gaussian-density}
\psi(t,x)=\frac{1}{\sqrt{2\pi t}}e^{-\frac{x^2}{2t}}.
\end{equation}

\begin{theorem}\label{thm:main-theorem}
Consider $X_\e$ defined by \eqref{eq:SDE} with initial condition $X_\e(0)=\e x$ and let $K(\eps)$ be any function that satisfies
\begin{equation}\label{eq:K-criterion}
\lim_{\e\downarrow 0}\e^{\gamma}K(\e)=0,\qquad \forall \gamma>0.
\end{equation}
Then, for all $\alpha>1$ and all $t\in\R$,
\begin{equation}\label{eq:main-thm-claim-1}
\lim_{\e\downarrow 0}\sup_{|x|\leq K(\e)}\left|\e^{-(\alpha-1)}\Pp\left(\tau_{\cI}^\e>\frac{\alpha}{\lambda}\log\e^{-1}+t;\ X_\e(\tau_{\cI}^\e)=q_{\pm}\right)-e^{-\lambda t}|f(q_\pm)|\psi_0(x)
\right|
=0,
\end{equation}
where
\[
\psi_0(x)=\psi\left(\frac{\sigma^2(0)}{2\lambda},x\right)=\sqrt{\frac{\lambda}{\pi}}\frac{e^{-\lambda\left(\frac{x}{\sigma(0)}\right)^2}}{\sigma(0)}.
\]
In particular, for any $T\in\R$,
\begin{multline*}
\textsc{Law}\left[\left(\tau_{\cI}^\e-\frac{\alpha}{\lambda}\log\e^{-1}-T,\ X_\e(\tau_{\cI}^\e)\right)\bigg|\tau_{\cI}^\e>\frac{\alpha}{\lambda}\log\e^{-1}+T\right]\\ \Rightarrow \exp_\lambda\otimes\left(\frac{|f(q_-)|}{|f(q_-)|+|f(q_+)|}\delta_{q_-}+\frac{|f(q_+)|}{|f(q_-)|+|f(q_+)|}\delta_{q_+}\right),
\end{multline*}
where $\Rightarrow$ stands for weak convergence, and   $\exp_{\lambda}$ is the exponential distribution with rate $\lambda>0$, i.e., $\exp_\lambda[t,\infty)=e^{-\lambda t}$ for $t\ge 0$.
\end{theorem}
\begin{remark}\rm 
We say that a function satisfying \eqref{eq:K-criterion}  \textit{grows at most subpolynomially} at $0$. We say that a function $c(\e)$ \textit{decays at most subpolynomially} at $0$ if $1/c(\e)$ grows subpolynomially at $0$. For brevity, we will usually omit the reference to $0$ and simply say  \textit{grows/decays subpolynomially} even when the function might not actually grow.
\end{remark}
\begin{remark}\rm
The theorem is stated for initial conditions that are at most of the order of $\eps$ away from the origin up to a subpolynomial factor. The case of initial conditions of the order of $\eps^\beta$ for $\beta<1$ is less interesting since then the tails of exit times decay as stretched exponentials of $\eps^{-1}$ instead of the power decay given by~\eqref{eq:main-thm-claim-1} (see Proposition \ref{prop:LDPestim}).
\end{remark}
\begin{remark}\rm  In the statement of Theorem~\ref{thm:main-theorem} and in the sequel, we adopt the usual convention  that each relation involving $\pm$ and $\mp$ stands for two relations, one with all top signs and one with all all bottom signs.
\end{remark}

The brief outline of our approach to the proof of this theorem is as follows. It is convenient to work in coordinates given by the function $f$ defined in \eqref{eq:coord-change-diffeo} where the drift is linear. We study the dynamics of the linear process in two separate phases: (1) in a neighborhood of the critical point of radius $\e^{\beta}$ for $\beta\in(0,1)$; (2) between leaving this small neighborhood and reaching the boundary of $f(\cI)$. 

In the second stage, the drift dominates the noise, and the process closely follows the corresponding deterministic trajectory one obtains by setting $\e=0$. The outcome of the first stage, i.e., the
exit from $[-\e^{\beta},\e^{\beta}]$, is determined though by a delicate interplay between the noise and the drift in an even smaller neighborhood of the origin ($\beta$ can be chosen arbitrarily close to one). We study this regime by introducing an auxiliary process $Z_\e(t)$ with constant diffusion coefficient approximating $Y_\e(t)=f(X_\e(t))$ pathwise at least over time intervals that are not too large and for which Theorem \ref{thm:main-theorem} is easier to establish.  Since $Z_\e(t)$ and $Y_\e(t)$ do not, in general, stay close on longer timescales, we introduce an iterative scheme to tackle this problem. Namely, we split the longer time interval into shorter ones and show that a useful approximation result, which holds under conditioning on the process not having exited the spatial interval, can be applied sequentially.

The plan of the paper is as follows. In Section \ref{sec:proof-of-main}, we perform the aforementioned change of variables to linearize the drift and prove Theorem \ref{thm:main-theorem} using an intermediate result on the exit from a small neighborhood of the origin. In Section \ref{sec:linear-additive}, we introduce an auxiliary process, which is fully linear and thus allows us to derive certain properties of the exit problem through explicit calculations. In Section \ref{sec:short-time}, we prove an approximation result which allows us to transfer these properties from the fully linear process to the case where only the drift is linear as long as the timescales involved are not too large. Finally, in Section \ref{sec:extension}, we use an iterative scheme to lift this limitation thereby finishing the proof of the intermediate result. In Section~\ref{sec:heteroclinic}, we explain how the result of this paper fits our program on long-term behavior of diffusions near heteroclinic networks.

{\bf Acknowledgment.} Yuri Bakhtin gratefully acknowledges partial support from NSF via grant
DMS-1811444. 

\section{Proof of Theorem \ref{thm:main-theorem}}\label{sec:proof-of-main}
As outlined above, we study the system first in a small neighborhood of the origin and then  after the process has escaped this small neighborhood. 

Let us start with the first part. The diffeomorphism 
$f:\cI\to\R$ introduced in \eqref{eq:coord-change-diffeo} and its inverse $g=f^{-1}$ 
provide a conjugation between the flow~$(S^t)$ and a linear flow:
\begin{equation}\label{eq:conjugation}
f(S^tx)=e^{\lambda t}f(x),\qquad\textrm{or}\qquad f'(x)b(x)=\lambda f(x).
\end{equation}
Note that the integrand in \eqref{eq:coord-change-diffeo} is quadratic when $x$ is close to zero and thus we have $f(0)=0$ and $f'(0)=1$. Outside of $\cI$, we define $f$ so that $f'$ and $f''$ are bounded.

Let $Y_\e(t)=f(X_\e(t))$ for times prior to the escape from $\cI$.
It\^o's formula and \eqref{eq:conjugation} then imply that this process satisfies the stochastic differential equation
\begin{equation}\label{eq:lin-SDE}
dY_\e(t)=\lambda Y_\e(t)dt+\e\tilde\sigma(Y_\e(t))dW(t)+\frac{\e^2}{2}h(Y_\e(t))dt
\end{equation}
for $t<\tau_{\cI}^\e$, where $\tilde\sigma(y)=f'(g(y))\sigma(g(y))$ and $h(y)=f''(g(y))\sigma^2(g(y))$.
 Due to boundedness of $f'$ and $f''$, $\tilde\sigma$ and $h$ are also bounded. 
 
 By Duhamel's formula, $Y_\e$ satisfies the integral equation
\begin{equation}\label{eq:Y-Duhamel}
Y_\e(t)=e^{\lambda t}\left(Y_\e(0) + \e U_\e(t)+\e^2 V_\e(t)\right),
\end{equation}
where
\[
U_\e(t)=\int_0^te^{-\lambda s}\tilde\sigma(Y_\e(s))dW(s),\qquad V_\e(t)=\frac{1}{2}\int_0^te^{-\lambda s}h(Y_\e(s))ds.
\]
Due to our conventions on $f',f''$ outside of $\cI$, the processes $U_\eps(t)$ and $V_\eps(t)$ are defined for all $t\ge 0$. Moreover, 
the boundedness of $h$ immediately implies the boundedness of $V_\e(t)$:
\begin{equation}
\label{eq:boundedness_of_V_e}
\|V_\eps(\cdot)\|_\infty\le \frac{\|h\|_\infty}{2\lambda},
\end{equation}
where $\|\cdot\|_\infty$ is the sup-norm on $[0,\infty)$. The boundedness of $\tilde\sigma$ yields a similar conclusion about the quadratic variation of $U_\e(t)$. Hence,  the existence of constants $c_1,c_2, N_0>0$ such that 
\begin{equation}\label{eq:exp-martingale-ineq}
\Pp\left(\sup_{t\ge 0}|U_\e(t\wedge\tau_\cI^{\e})|\geq N\right)\leq c_1e^{-c_2N^2},\qquad N\geq N_0,
\end{equation}
is implied by the following exponential martingale inequality (see, e.g., Problem~12.10 in~\cite{Bass:MR2856623}):
\begin{lemma}\label{lem:exp-marting-ineq}
Let $M(t)$ be a centered martingale with quadratic variation process $\langle M\rangle_t$. Then
\[
\Pp\left(\sup_{t\ge 0}|M(t)|\geq a; \langle M\rangle_{\infty}\leq b\right)\leq 2e^{-\frac{a^2}{2b}}.
\]
\end{lemma}
 
 Let us take $\beta\in(0,1)$ and set 
 $\mathcal{V}=g\left([-\e^{\beta},\e^{\beta}]\right)\subseteq \cI$. 
The following result describes the tail behavior of  $\tau_\cV^\e$, the exit time from~$\mathcal{V}$. In particular, it says that, in the $\e\downarrow 0$ asymptotics, the choice of the exit direction is distributed symmetrically independently of the exit time.
 \begin{theorem}\label{thm:linear}
Let $Y_\e(0)=\e y$, where $|y|\leq K(\e)$ with $K(\e)$ growing subpolynomially at $0$. Then, for all $\alpha>1$, $C\in\R$, and any function $c(\eps)$ satisfying $\lim_{\eps\to 0}c(\eps)=0$, there is $\beta_0\in(0,1)$ such that for $\beta\in(\beta_0,1)$, we have
\begin{equation}\label{eq:linear-exit-time-tail}
\lim_{\e\downarrow 0}\sup_{|y|\leq K(\e)}\left|\e^{-(\alpha-1)}\Pp\left(\tau_{\cV}^\e>\frac{\alpha-\beta}{\lambda}\log\e^{-1}-C+c(\e); Y_\e(\tau_\cV^\e)=\pm\e^{\beta}\right)-
e^{\lambda C}\psi_0(y)
\right|
=0.
\end{equation}

\end{theorem}

We give the proof of Theorem~\ref{thm:linear} in Section~\ref{sec:extension}. 

After exit from $\cV$, the deterministic dynamics dominates the evolution, which means that the exit time will be close to 
\begin{equation}\label{eq:determ-times}
T_\e^\pm:=T\left(g\left(\pm \e^{\beta}\right)\right)=\frac{\beta}{\lambda}\log\e^{-1}+\frac{1}{\lambda}\log |f(q_\pm)|,
\end{equation}
the time it takes for $X_0(t)$ to exit $\cI$ starting at $g\left(\pm \e^{\beta}\right)$.
This is captured by the following standard large deviation estimates.

\begin{proposition}\label{prop:LDPestim}
Let $X_{\e}(0)=g\left(\pm \e^{\beta}\right)$. Then for every $\beta'\in(0,\beta)$ and subpolynomially decaying function $c(\e)>0$, 
we have
\begin{equation}\label{eq:LDP-for-time}
 \Pp\left(\left|\tau_\cI^\e-T_\e^\pm\right|>c(\eps)\right)=o\left(e^{-\frac{1}{\eps^{2(1-\beta')}}}\right)
\end{equation}
and
\begin{equation}\label{eq:LDP-for-direction}
\Pp\left(X_\e(\tau_\cI^\e)=q_{\pm}\right)\geq 1-o\left(e^{-\frac{1}{\eps^{2(1-\beta')}}}\right)
\end{equation}
\end{proposition}

\bpf 
We start by showing that with overwhelming probability the exit happens through the endpoint that is on the same side as the starting point. Indeed, \eqref{eq:Y-Duhamel},~\eqref{eq:boundedness_of_V_e}, and~\eqref{eq:exp-martingale-ineq} imply
\begin{align*}
\Pp\left(X_\e(\tau_\cI^\e)=q_{\mp}\right)&=\Pp\left(Y_\e(\tau_\cI^\e)=f(q_{\mp})\right)=\Pp\left(\eps^\beta <\left|\eps U_\e(\tau_\cI^\e)+\e^2V_\e(\tau_\cI^\e)\right|\right)
\\&\leq \Pp\left(\eps^\beta <\eps\sup_{t>0}\left| U_\e(t\wedge\tau_\cI^\e)\right|+\e^2\frac{\|h\|_{\infty}}{2\lambda}\right)
\\&\leq\Pp\left(\eps^{-(1-\beta)}-\e\frac{\|h\|_{\infty}}{2\lambda} <\sup_{t>0}\left| U_\e(t\wedge\tau_\cI^\e)\right|\right)=o\left(e^{-\frac{1}{\eps^{2(1-\beta')}}}\right),
\end{align*}
and \eqref{eq:LDP-for-direction} follows.

To prove \eqref{eq:LDP-for-time}, let us introduce
\[
U_\e =U_\e\left((T_\e^\pm+c(\e))\wedge\tau_\cI^\e\right),\qquad V_\e =V_\e\left((T_\e^\pm+c(\e))\wedge\tau_\cI^\e\right),
\]
and note that \eqref{eq:Y-Duhamel} implies
\begin{align*}
\Pp\left(\tau_\cI^\eps>T_\eps^{\pm}+c(\eps)\right) &\le \Pp\left(f(q_-)< e^{\lambda \left(T_\eps^{\pm}+c(\eps)\right)}\left(\pm\eps^\beta + \eps U_\e+\e^2V_\e\right) <f(q_+)\right)
\\ &= \Pp\left(\eps^\beta e^{-\lambda c(\eps)}\frac{f(q_-)}{|f(q_\pm)|} <\pm\eps^\beta + \eps U_\e(t)+\e^2V_\e(t)<\eps^\beta e^{-\lambda c(\eps)}\frac{f(q_+)}{|f(q_\pm)|}\right)
\\&\le \Pp\left( \left|U_\e\right|>\e^{-(1-\beta)}\left(1-e^{-\lambda c(\eps)}\right)-\e\frac{\|h\|_\infty}{2\lambda}\right) +  o\left(e^{-\frac{1}{\eps^{2(1-\beta')}}}\right)
\\&=o\left(e^{-\frac{1}{\eps^{2(1-\beta')}}}\right),
\end{align*} 
where we used  \eqref{eq:determ-times}, \eqref{eq:LDP-for-direction}, \eqref{eq:boundedness_of_V_e}, \eqref{eq:exp-martingale-ineq}, and  the subpolynomial decay of $c(\e)$. Similarly,
\begin{align*}
\Pp\left(\tau_\cI^\eps<T_\eps^{\pm}-c(\eps)\right)
&\le \Pp\left(\sup_{t<T_\eps^{\pm}-c(\eps)}e^{\lambda (t-c(\eps))}\left|\pm\eps^\beta + \eps U_\e(t)+\e^2V_\e(t)\right| > |f(q_{\pm})|\right)+ o\left(e^{-\frac{1}{\eps^{2(1-\beta')}}}\right)
\\&\le \Pp\left(e^{\lambda (T_\eps^{\pm}-c(\eps))}\left(\eps^\beta + \eps \sup_{t<T_\eps^\pm} \left|U_\e(t)\right|+\e^2\frac{\|h\|_{\infty}}{2\lambda}\right) > |f(q_\pm)|\right)+ o\left(e^{-\frac{1}{\eps^{2(1-\beta')}}}\right)
\\&= \Pp\left(\eps^\beta + \eps \sup_{t<T_\eps^{\pm}} \left|U_\e(t)\right|+\e^2\frac{\|h\|_{\infty}}{2\lambda}\ge \eps^\beta e^{\lambda c(\eps)}\right) +  o\left(e^{-\frac{1}{\eps^{2(1-\beta')}}}\right)
\\&\leq \Pp\left( \sup_{t>0} \left|U_\e(t\wedge\tau_\cI^\e)\right|\ge \eps^{-(1-\beta)} \left(e^{\lambda c(\eps)}-1\right)-\e \frac{\|h\|_{\infty}}{2\lambda}\right)+ o\left(e^{-\frac{1}{\eps^{2(1-\beta')}}}\right)
\\&=o\left(e^{-\frac{1}{\eps^{2(1-\beta')}}}\right).
\end{align*}
\epf

\noindent\textit{Proof of Theorem \ref{thm:main-theorem}.}
We have
\[
Y_\e(0)= f(X_\e(0))=f(\e x)=\e x + \mathcal{O}(\e^2 x^2)=\e(x+\mathcal{O}(\e x^2))
\]
and 
\[
|x+\mathcal{O}(\e x^2)|\leq K(\e)+\mathcal{O}\left(\e (K(\e))^2\right)\leq 2K(\e)
\]
for small enough $\e$, so the right-hand side is a function subpolynomially growing at $0$.  
Let us define $\theta^{\pm}_\e=T^\pm_\eps-(\tau_\cI^\e-\tau_\cV^\e)$. 
Due to \eqref{eq:LDP-for-direction},
\begin{align*}
\Pp&\left(X_\e\left(\tau_{\cI}^\e\right)=q_\pm;\tau_{\cI}^\e>\frac{\alpha}{\lambda}\log\e^{-1}{+t}\right)=\Pp\left(Y_\e\left(\tau_{\cV}^\e\right)=\e^{\beta};\tau_{\cV}^\e>\frac{\alpha}{\lambda}\log\e^{-1}{+t}-T^\pm_\eps+\theta^{\pm}_\e\right)+o\left(\e^{\alpha-1}\right),
\end{align*}
where the error term (along with all subsequent error terms) is uniform in the starting points $|x|\leq K(\e)$. Note that the strong Markov property implies the conditional independence of $\theta^{\pm}_\e$ and $\tau_\cV^\e$ given $Y_\e(\tau_\cV^\e)=\e^{\beta}$. This, along with \eqref{eq:LDP-for-time},  allows us to give upper and lower estimates of the first term on the right-hand side:
\begin{multline*}
\Pp\left(Y_\e\left(\tau_{\cV}^\e\right)=\pm\e^{\beta};\tau_{\cV}^\e>\frac{\alpha}{\lambda}\log\e^{-1}{+t}-T^\pm_\eps+c(\e)\right) + o(\e^{\alpha-1})\\
\leq
\Pp\left(Y_\e\left(\tau_{\cV}^\e\right)=\e^{\beta};\tau_{\cV}^\e>\frac{\alpha}{\lambda}\log\e^{-1}{+t}-T^\pm_\eps+\theta_\e^{\pm}\right)\\
\leq\Pp\left(Y_\e\left(\tau_{\cV}^\e\right)=\pm\e^{\beta};\tau_{\cV}^\e>\frac{\alpha}{\lambda}\log\e^{-1}{+t}-T^\pm_\eps-c(\e)\right) + o(\e^{\alpha-1}).
\end{multline*}
where $c(\e)$ is an arbitrary positive function that decays subpolynomially. Now we may apply Theorem \ref{thm:linear} with $C={-t} +\lambda^{-1}\log|f(q_\pm)|$ to both sides and conclude the proof.
\qed

\section{Linear system with additive noise}\label{sec:linear-additive}
In this section, we introduce an auxiliary process, which is a simpler special case of \eqref{eq:lin-SDE}. Namely, we consider
\begin{align}
\label{eq:linear-additive-sde}
dZ_\eps(t)&=\lambda Z_\eps(t)dt + \eps \sigma_0 dW(t),\qquad Z_\eps(0)=\e z,
\end{align}
where $\sigma_0=\tilde{\sigma}(0)=\sigma(0)>0$, $|z|<K(\eps)$, and $K(\eps)$ grows subpolynomially at $0$.

We will need a precise description of the exit of $Z_\e$ from $\cV_Z=[-\e^{\beta}(1+\delta_\e),\e^{\beta}(1+\delta_\e)]$ for $\beta\in (0,1)$ and any $\delta_\eps>0$ satisfying  $\delta_\e\downarrow 0$ as $\e\downarrow 0 $.  Let us introduce a stopping time 
\[
\tau_\eps=\inf\{t>0:\ |Z_\eps(t)|=\eps^\beta(1+\delta_\e)\},
\]
and, for a constant $C$ and  a subpolynomially decaying at $0$  function $c(\eps)$,  a deterministic time
\begin{equation}\label{eq:t_eps}
t_\e=\frac{\alpha-\beta}{\lambda}\log\e^{-1}-C+c(\eps).
\end{equation}
We will often use $C_\e=C-c(\e)$.
The first result of this section is a version of Theorem \ref{thm:linear} for the process $Z_\eps$ with stronger control of the dependence on the initial point.
\begin{lemma}\label{lem:main-beta-theorem-for-linear}
For any $\alpha>1$ and any subpolynomially decaying function $c(\e)$ in the definition of $t_\eps$, there is $c>0$ such that
\[
\lim_{\e\downarrow 0}\sup_{|z|\leq K(\e)}e^{cz^2}\left|\e^{-(\alpha-1)}\Pp\left(\tau_\e>t_\e\right)-2e^{\lambda C}\psi_0(z)\right|
=0.
\]
\end{lemma}

\begin{proof}
Duhamel's formula gives an explicit solution to \eqref{eq:linear-additive-sde}:
\begin{equation}
Z_\eps(t)=\eps e^{\lambda t} (z+\sigma_0 N(t)),
\label{eq:duhamel}
\end{equation}
where $N(t)=\int_0^t e^{-\lambda s} dW(s)$.
Plugging in $\tau_\e$ for $t$, we obtain 
\begin{equation}\label{eq:exit-time-def-eq}
(1+\delta_\e)\e^{\beta}=|Z_\eps(\tau_\e)|=\eps e^{\lambda \tau_\e} \left|z+\sigma_0N(\tau_\e)\right|,
\end{equation}
which is equivalent to
\[
\tau_\e=\frac{1-\beta}{\lambda}\log\e^{-1}-\frac{1}{\lambda}\log|z+\sigma_0N(\tau_\e)|+o(1).
\]
Therefore,
\[
\left\{\tau_\e> t_\e\right\}=\left\{|z+\sigma_0N(\tau_\e)|< e^{\lambda C_\eps}\e^{\alpha-1}(1+o(1))\right\}.
\]

By the martingale convergence theorem, $N(t)$ converges to a random variable $N_{\infty}$ as $t\to\infty$ almost surely and in $L^1$. 
We claim that, in addition, there are $c_1,c_2>0$ such that for any $L>0$,
\begin{equation}
\label{eq:discrepancy-between-N_infty-and-N_stopped}
\Pp\left(|N(\tau_\e)-N_\infty|>L\e^{\alpha-\beta}; \tau_\e>t_\e\right)\leq c_1 e^{-c_2 L^2}.
\end{equation}
Indeed, Lemma~\ref{lem:exp-marting-ineq} implies
\[
\Pp\left(\sup_{t\geq t_\e}|N(t)-N_\infty|\geq L\e^{\alpha-\beta}\right)\leq
2\Pp\left(\sup_{t\geq t_\e}|N(t)-N(t_\e)|\geq L\e^{(\alpha-\beta)}/2\right)
\leq c_1e^{-c_2L^2\e^{2(\alpha-\beta)}e^{2\lambda t_\e}}
\]
for some $c_1,c_2>0$, which proves the claim since $
\e^{\alpha-\beta}e^{\lambda t_\e}=e^{-\lambda C_\eps}.$

Let us fix $\gamma\in (\alpha-1,\alpha-\beta)$ and write
\[
\Pp\left(\tau_\e> t_\e\right)=H_1(z,\e)+H_2(z,\e),
\]
where
\begin{align*}
H_1(z,\e)&=\Pp\left(|z+\sigma_0N_\infty+\sigma_0(N(\tau_\e)-N_{\infty})|< e^{\lambda C_\eps }\e^{\alpha-1}(1+o(1)); |N(\tau_\e)-N_{\infty}|\le \e^{\gamma} \right),\\
H_2(z,\e)&=\Pp\left(|z+\sigma_0N(\tau_\e)|< e^{\lambda C_\eps}\e^{\alpha-1}(1+o(1));|N(\tau_\e)-N_{\infty}|>\e^{\gamma}\right).
\end{align*}
The random variables $N(\tau_\e)$ and $N(\tau_\e)-N_\infty$ are independent due to the strong Markov property. So the Gaussian tail of the maximum of the Brownian motion and
\eqref{eq:discrepancy-between-N_infty-and-N_stopped} imply
\[
H_2(z,\e)\leq c_1e^{-c_2z^2}o(\e^{\alpha-1}).
\]
The desired asymptotics of $H_1(z,\eps)$ follows from  the explicit form of the Gaussian density of 
the random variable $z+\sigma_0 N_\infty$.
\end{proof}

The next result is based on the fact that the distribution of $Z_\e(t_\e)$ conditioned on non-exit is approximately uniform over $[-\e^{\beta}(1+\delta_\e),\e^{\beta}(1+\delta_\e)]$. In fact, this stronger statement is proved as an intermediate step. We first note that the density of any absolutely continuous random variable conditioned on a positive probability event is well-defined.
\begin{lemma}\label{lem:X-equid}
If $f_\eps^c(u)$ is the probability density of $Z_\eps(t_\eps)$ conditioned on $\{\tau_\eps > t_\eps\}$, then for any $\delta>0$
\begin{equation}
\lim_{\e\downarrow 0}\sup_{|u|\leq(1-\delta)\e^{\beta}}\sup_{|z|\leq K(\e)}\left|\eps^{\beta}f_\eps^c(u)- \frac{1}{2}\right|=0.
\label{eq:equidistribution} 
\end{equation}
 Moreover, for any integrable function $h$ with exponentially decaying tails and any subpolynomially growing function $K(\e)$,
\begin{equation}
\lim_{\e\downarrow 0}\sup_{|z|\leq K(\e)}\left|\e^{-(1-\beta)}\E\left[h\left(\e^{-1}Z_{\e}(t_\e)\right)|\tau_\e>t_{\e}\right]-\frac{1}{2} \int_{-\infty}^{\infty} h(y)dy\right|=0.
\label{eq:equidistribution-test-function}
\end{equation}
\end{lemma}

\begin{proof}
By~\eqref{eq:duhamel} and the Dambis--Dubins--Schwartz theorem (see, e.g., \cite[Section~3.4B]{KS1991}),
\begin{equation}
Z_\eps(t)=\eps e^{\lambda t} \left(z+\sigma_0\int_0^t e^{-\lambda s} dW(s)\right)=\eps e^{\lambda t} \left(z+B(r(t))\right), 
\label{eq:duhamel2}
\end{equation}
where $r(t)=\sigma_0^2\left(1-e^{-2\lambda t}\right)/(2\lambda)$, and $B$ is an auxiliary standard Brownian motion. 
Since
\[
r(t_\eps)=
\sigma_0^2\frac{1-\eps^{2(\alpha-\beta)}e^{2\lambda C_\e}}{2\lambda},
\]
we have
\begin{align*}
Z_\eps(t_\eps)=\eps^{1-\alpha+\beta}e^{-\lambda C_\e}\left(z+B\left(\sigma_0^2\frac{1-\eps^{2(\alpha-\beta)}e^{2\lambda C_\e}}{2\lambda}\right)\right).
\end{align*}
This 
is a Gaussian random variable. Its density at a point $u\in\R$ is given by
\[
 p_\eps(u)=\frac{\sqrt{\lambda}e^{\lambda C_\eps} }{\sqrt{\pi}\sigma_0 \eps^{1-\alpha+\beta}(1+o(1))}\exp\left\{-\frac{(u\eps^{\alpha-\beta-1} e^{\lambda C_\eps}-z)^2}{\frac{\sigma^2_0}{\lambda}(1+o(1))}\right\}.
\]
If $|u|\le\eps^\beta$, then $|u\eps^{\alpha-1-\beta}|\le \eps^{\alpha-1}$, so
\begin{equation}
  p_\eps(u)=\eps^{\alpha-\beta-1} e^{\lambda C} \psi_0(z)(1+o(1)),
\label{eq:Gaussian-density-asymptotics} 
\end{equation}
uniformly over $u$ and $z$ satisfying $|u|\le\eps^\beta$, $|z|\leq K(\e)$.

Therefore,~\eqref{eq:equidistribution} will follow from
 \begin{equation}
\label{eq:cond-density-asymp-uniform3}
\lim_{\e\downarrow 0}\sup_{|u|\leq (1-\delta)\e^\beta}\sup_{|z|\leq K(\e)}\left|\frac{f_\eps(z,u)}{\eps^{\alpha-\beta-1}e^{\lambda C}\psi_0(z)}-1\right|=0,
\end{equation}
where $f_\eps(z,u)$ is the sub-probability density of $Z_\e(t_\e)$ on the event $\{\tau_\e>t_\e\}$.  For this, it suffices to see that
\begin{equation}
\label{eq:influence_of_boundary3}
\sup_{|u|\leq (1-\delta)\e^\beta}\sup_{|z|\leq K(\e)}\frac{g_\eps(z,u)}{\psi_0(z)},
\end{equation}
decays exponentially fast as $\eps\downarrow 0$, where $g_\eps(z,u)$ is the sub-probability density of $Z_\eps(t_\eps)$ on the event $\{\tau_\eps \le t_\eps\}$.
Given that $Z_\eps(\tau_\e)=\e^{\beta}(1+\delta_\e)$  (similarly for $-\e^{\beta}(1+\delta_\e)$), we have
\[
Z_\eps(t_\eps)=e^{\lambda (t_\eps-\tau_\e)}\left(\e^{\beta}(1+\delta_\e)+\eps \sigma_0 \int_0^{t_\eps-\tau_\e}e^{-\lambda s}dW(s)\right).
\]
and thus
\begin{equation}
\label{eq:computing_h3}
g_\eps(z,u)=\int_0^{t_\eps}\Pp\{\tau_\eps\in [t,t+dt)\} 
G(\eps,t, u), 
\end{equation}
where
\begin{align*}
G(\eps,t,u)&=
\psi\left(\eps^2\sigma_0^2\frac{w^2-1}{2\lambda}, w\e^{\beta}(1+\delta_\e)-u\right)
= \frac{1}{\sqrt{2\pi \eps^2\sigma_0^2\frac{w^2-1}{2\lambda}}}e^{-\frac{\lambda(w\e^{\beta}(1+\delta_\e)-u)^2}{\eps^2\sigma_0^2(w^2-1)}},
\end{align*}
with $\psi(\cdot,\cdot)$ that was introduced in~\eqref{eq:Gaussian-density} and $w=e^{\lambda(t_\eps-t)}$.

We claim that there is $c>0$ such that
\begin{equation}
\label{eq:denisty-of-coming-back-from-the-boundary3}
\sup\{G(\eps,t,u):\ 0\le t\le t_\eps,\ |u|\le(1-\delta) \e^{\beta} \}=\mathcal{O}\left(e^{-c/\eps^{2(1-\beta)}}\right). 
\end{equation}
 This, along with \eqref{eq:computing_h3} and the fact that $\psi_0(z)>\lambda^{1/2}\pi^{-1/2}\sigma_0^{-1}e^{-\lambda K^2(\eps)/\sigma_0^2}$ for all
$\eps$ and $z$ satisfying $|z|<K(\eps)$, will imply the desired exponential decay in~\eqref{eq:influence_of_boundary3}.

Let us fix any $w_0$ and find $c_0>0$ such that 
\[
\lambda (w\e^{\beta}(1+\delta_\e)-(1-\delta)\e^\beta  )^2/(\sigma_0^2(w^2-1))\ge c_0\e^{2\beta}
\]
for $w> w_0$ and all $\eps>0$. Then  there is a constant $c_0$ such that for all $|u|\le (1-\delta)\eps^\beta$ and  $\eps>0$,
\begin{equation}
\label{eq:for_large_z3}
G(\eps,t,u)\le  \frac{1}{\sqrt{2\pi \eps^2\sigma_0^2\frac{w_0^2-1}{2\lambda}}}e^{-c_0/\eps^{2(1-\beta)}},\quad w>w_0.
\end{equation}
If $1 \le w \le w_0$, then $(w\e^{\beta}(1+\delta_\e)-u)^2\ge\Delta:= \e^{2\beta}\delta^2$. So,
denoting $D=\eps^2\sigma_0^2(w^2-1)/(2\lambda)$, we obtain
\begin{equation}
\label{eq:for_small_z__13}
G(\eps,t,u)\le \frac{1}{\sqrt{2\pi D}}e^{-\frac{\Delta}{2D}}=\psi(D, \Delta),\quad 1\le w\le w_0.
\end{equation}
The restriction on $w$ implies $0\le D \le c_1\eps^2$ for some $c_1$. To maximize $\psi(D,\Delta)$,
we compute
\[
\partial_D\ln \psi(D,\Delta)= -\frac{1}{2D} + \frac{\Delta}{2 D^2},
\]
so $\psi(D,\Delta)$ grows in $D\in [0,\Delta]$, and we obtain from \eqref{eq:for_small_z__13}:
\begin{equation}
\label{eq:for_small_z__23}
G(\eps,t,x)\le \frac{1}{\sqrt{2\pi c_1 \eps^2}}e^{-\frac{\Delta}{2c_1\eps^2}}.
\end{equation}
Combining \eqref{eq:for_large_z3} and \eqref{eq:for_small_z__23}, we obtain \eqref{eq:denisty-of-coming-back-from-the-boundary3} and hence \eqref{eq:influence_of_boundary3}.
Thus, \eqref{eq:equidistribution} is proved.
 
To prove \eqref{eq:equidistribution-test-function}, we write
\[
\E\left[h(\e^{-1}Z_{\e}(t_\e))|\tau_\e>t_{\e}\right] = 
\int_{|u|<(1-\delta)\e^{\beta}}h(\e^{-1}u)f_\e^c(u)du + \int_{(1-\delta)\e^{\beta}
\le|u|\le (1+\delta_\eps)\eps^\beta}h(\e^{-1}u)f_\e^c(u)du
\]
and notice that the first term on the right-hand side equals
\begin{multline*}
\int_{-(1-\delta)\e^{\beta}}^{(1-\delta)\e^{\beta}}h(\e^{-1}x)f_\e^c(x)dx=\frac{1}{2\e^{\beta}}(1+o(1))\int_{-(1-\delta)\e^{\beta}}^{(1-\delta)\e^{\beta}}h(\e^{-1}x)dx\\
= \frac{1}{2}\e^{1-\beta}(1+o(1))\int_{-(1-\delta)\e^{-(1-\beta)}}^{(1-\delta)\e^{-(1-\beta)}}h(x)dx=\frac{1}{2}\e^{1-\beta}\int_{-\infty}^{\infty}h(x)dx+o(\e^{1-\beta}),
\end{multline*}
while the second term decays much faster than $\eps^{1-\beta}$ due to the decay assumption on $h$.
\end{proof}

\section{Proof of Theorem \ref{thm:linear} for $\alpha\in(1,1+\beta)$}\label{sec:short-time}
We start by studying the deviations $\Delta_\eps(t)=Y_\eps(t)-Z_\eps(t)$ as long as both processes $Y_\eps(t)$ and $Z_\eps(t)$ are close to the origin. Let us fix $\beta\in(0,1)$ and introduce the stopping times
\[
\tauy_\eps=\inf\{t\ge 0:\ |Y_\eps(t)|=\eps^\beta \},\qquad
\ttau_\eps=\inf\{t\ge 0:\ |Y_\eps(t)|=2\eps^\beta \}.
\]
\begin{lemma}\label{lem:Delta-estimate}
Suppose $\alpha\in(1,1+\beta)$, $\beta'\in(\alpha-1,\beta)$, $L(\e)>0$ is a bounded function,  and
\begin{equation}\label{eq:t-eps-prime}
t_\epsilon'=\frac{\alpha-\beta}{\lambda}\log \left(L(\e)\e^{-1}\right).
\end{equation}
Then for sufficiently small $\eps>0$,
we have
\begin{equation}\label{eq:deviation-prob}
\Pp\left(\sup_{0\le t\le t_\eps' \wedge \ttau_\eps} |\Delta_\eps(t)|>\eps^{\beta+\beta'-(\alpha-1) } \right)\leq 2e^{-\frac{c}{\eps^{2(\beta-\beta')}}}, 
\end{equation}
provided $Y_\e(0)=Z_\e(0)=\eps y$, for every $y\in(-2\eps^\beta, 2\eps^\beta)$. 
\end{lemma}

\begin{proof}
Recalling~\eqref{eq:Y-Duhamel}, we obtain
\[
\Delta_\eps(t)=\eps e^{\lambda t} \left(I_\eps^{(1)}(t)+I_\eps^{(2)}(t)\right),\qquad
I_\eps^{(1)}(t)=\int_0^t e^{-\lambda s}(\tilde\sigma(Y_\eps(s))-\sigma_0)dW(s),\qquad I^{(2)}_\e=\mathcal{O}(\e).
\]
Clearly, $I_\e^{(1)}$ is a martingale satisfying 
$\langle I_\eps^{(1)}\rangle_t=\mathcal{O}\left(\eps^{2\beta}\right)$ for  $t\le \ttau_\eps$,
so for any $c_0>0$, there is $c>0$ such that
\[
\Pp\left(\sup_{s\le \ttau_\eps}|I_\eps^{(1)}|>c_0\eps^{\beta'}\right)\le 2 e^{-c\frac{\eps^{2\beta'}}{\eps^{2\beta}}}\le 2e^{-c\frac{1}{\eps^{2(\beta-\beta')}}},
\]
by Lemma~\ref{lem:exp-marting-ineq}. Therefore
\[
\Pp\left(\sup_{s\le \ttau_\eps}|I_\eps^{(1)}+I_\eps^{(2)}|>2c_0\eps^{\beta'}\right)\le 2e^{-c\frac{1}{\eps^{2(\beta-\beta')}}},
\]
and on the complementary event we have
\[
\sup_{0\le t\le t'_\eps \wedge \ttau_\eps} |\Delta_\eps(t)|<2c_0 \eps L^{\alpha-\beta}(\eps) \eps^{-(\alpha-\beta)}\eps^{\beta'} < \eps^{\beta+\beta'-(\alpha-1) }
\]
if $c_0$ is chosen sufficiently small, which finishes the proof.
\end{proof}

Based on this approximation result and the calculation for $Z_\e$ in the previous section, the following theorem proves Theorem \ref{thm:linear} for $\alpha$ not too large.

\begin{theorem}\label{thm:Y-exits-with-X}
Let $\alpha\in(1,1+\beta)$ and let $t_\eps$ be as in \eqref{eq:t_eps}. Then there is $c>0$ such that
\begin{equation}
\label{eq:exit-eps-beta-Y-small-alpha}
\lim_{\e\downarrow 0}\sup_{|y|\leq K(\e)}e^{cy^2}\left|\e^{-(\alpha-1)}\Pp\left(\tauy_\e>t_\eps; 
Y_\eps(\tauy_\e)=\pm\e^{\beta}
\right)-e^{\lambda C}\psi_0(y)\right|=0.
\end{equation}
\end{theorem}

\begin{proof}
We start with an upper bound on $\Pp(\tauy_\e>t_\e)$ in terms of $Z_\eps$. Let us fix any
$\beta''\in(\beta,2\beta -(\alpha-1))$, so that $\beta'=\beta'' -\beta+\alpha -1\in(\alpha-1,\beta)$, which will allow us to apply Lemma~\ref{lem:Delta-estimate} several times in this proof. Let us take any family of events $(B_\eps)_{\eps>0}$ and estimate
\begin{multline*}
\Pp(\tauy_\eps>t_\e; B_\eps)=I_1(\eps)+I_2(\eps)
\\=
\Pp\left(\tauy_\eps>t_\e;\sup_{0\leq t\leq t_\e\wedge\ttau_\eps}|\Delta_\e(t)|<\e^{\beta''}; B_\eps\right)+\Pp\left(\tauy_\eps>t_\e;\sup_{0\leq t\leq t_\e\wedge\ttau_\eps}|\Delta_\e(t)|\geq\e^{\beta''}; B_\eps\right).
\end{multline*}
Note that $t_\e$ is of the form \eqref{eq:t-eps-prime} with $L(\eps)=e^{-\frac{(\alpha-\beta)C_\eps}{\lambda}}$ and thus \eqref{eq:deviation-prob} implies
\begin{equation}\label{eq:secterm2}
I_2(\eps)=o_{\exp}(1),
\end{equation}
where, for any $\gamma>0$, we use $o_{\exp}(1)$ as a shorthand for $o(e^{-\eps^{-\gamma}})$ . Also
\[
I_1(\eps)=\Pp\left(\tauy_\eps>t_\e;\sup_{0\leq t\leq t_\e}|\Delta_\e(t)|<\e^{\beta''}; B_\eps\right).
\]
We need to approximate this in terms of the exit time of $Z_\eps$ instead of~$\tauy_\eps$. We do not have control over the difference of these two times in general as we can only control the difference of the processes until $t_\eps$. Instead, we are going to set a different threshold for $Z_\eps$ to reach. Let
$\gamma\in (\beta,\beta'')$, $l_\e^1=\e^{\beta}$ and $l_\e^2=\l_\e^1+\eps^\gamma$.
This implies
\[
\Pp\left(\tauy_\eps>t_\e;\tau^Z_\e \leq t_\e; \sup_{0\leq t\leq t_\e}|\Delta_\e(t)|\leq\e^{\beta''}; B_\eps\right)=0,
\]
where $\tau^Z_\eps$ is the exit time from $[-l_\e^2,l_\e^2]$ for $Z_\eps$,
and thus
\[
I_1(\eps)
=\Pp\left(\tauy_\eps>t_\e;\tau^Z_\e >t_\e; \sup_{0\leq t\leq t_\e}|\Delta_\e(t)|\leq\e^{\beta''}; B_\eps\right)
\leq \Pp\left(\tau^Z_\eps>t_\e; B_\eps\right).
\]
Combining this with \eqref{eq:secterm2}, we obtain
\begin{equation}\label{eq:upper_bd2}
\Pp\left(\tauy_\eps>t_\e; B_\eps\right)\leq\Pp\left(\tau_\e^Z>t_\e; B_\eps\right)+o_{\exp}(1).
\end{equation}
Next, we set $l_\e^3=l_\e^{1}-\e^{\gamma}$ and define $\eta_\e^Z$ to be the exit time of $Z_\eps$ from $[-l_\e^3,l_\e^3]$. Lemma~\ref{lem:Delta-estimate} and the fact that $\{\tauy_\e\leq t_\e;\ \sup_{0\leq t\leq t_\e\wedge\ttau_\e}|\Delta_\e(t)|\leq\e^{\beta''}\}\subset \{\eta_\e^Z\le t_\e\}$, imply
\begin{align}\label{eq:lower_bd2}
\Pp(\tauy_\e>t_\e; B_\eps)&\geq\Pp\left(\tauy_\eps>t_\e;\sup_{0\leq t\leq t_\e\wedge\ttau_\e}|\Delta_\e(t)|\leq\e^{\beta''}; B_\eps\right) \\
\notag&=\Pp(B_\eps)+o_{\exp}(1)-\Pp\left(\tauy_\eps\leq t_\e;\sup_{0\leq t\leq t_\e\wedge\ttau_\e}|\Delta_\e(t)|\leq\e^{\beta''}; B_\eps\right)\\
\notag&\geq \Pp(B_\eps)+o_{\exp}(1)-\Pp\left(\eta_\e^Z\leq t_\e; B_\eps\right)=\Pp\left(\eta_\e^Z>t_\e; B_\eps\right)+o_{\exp}(1).
\end{align}
Combining \eqref{eq:upper_bd2} and \eqref{eq:lower_bd2}, we obtain
\begin{equation}\label{eq:stopping-time-sandwich}
\Pp(\eta_\e^Z>t_\e; B_\eps)-o_{\exp}(1)\leq\ \Pp(\tauy_\e>t_\e; B_\eps)\ \leq \Pp(\tau_\e^Z>t_\e; B_\eps)+o_{\exp}(1).
\end{equation}
This,  Lemma~\ref{lem:main-beta-theorem-for-linear}, and our choice of $l_\e^i$, $i=1,2,3$ imply
\begin{equation}
\label{eq:exit-eps-beta-Y-small-alpha-no-direction}
\lim_{\e\downarrow 0}\sup_{|y|\leq K(\e)}e^{cy^2}\left|\e^{-(\alpha-1)}\Pp\left(\tauy_\e>t_\eps
\right)-2e^{\lambda C}\psi_0(y)\right|=0.
\end{equation}

\medskip

To finish the proof, we will need
\begin{equation}\label{eq:side-switch-unlikely}
\lim_{\e\downarrow 0}\sup_{\e^{\beta''}\le |u| \le \e^{\beta}}\left|\Pp\left(Y_\eps(\tauy^{\e})=\e^{\beta}|Y_\e(t_\e)=u\right)-\ONE_{u>0}\right|=0,
\end{equation}
which holds since, due to \eqref{eq:Y-Duhamel}, $Y_\eps(\tauy_\e)=\e^{\beta}$ is equivalent to
$\e^{-1}u+U_\e(\tauy_\e)+\e V_\e(\tauy_\e)>0,$ and so 
\[
\left|\Pp\left(Y_\eps(\tauy^{\e})=\e^{\beta}|Y_\e(t_\e)=u\right)-\ONE_{u>0}\right|\leq\Pp\left(|U_\e(\tauy_\e)|\geq \e^{-(1-\beta'')}-\e V(\tauy_\e)\right)\to 0,\quad  \e\downarrow 0, 
\]
due to the boundedness of $V(\tauy_\e)$ and  \eqref{eq:exp-martingale-ineq}. 

Using \eqref{eq:side-switch-unlikely}, we can write
\begin{align}
&\Pp\left(Y(\tauy_\e)=\e^{\beta}|\ \tauy_\e>t_\eps\right)=\E\left[\Pp\left(Y_\eps(\tauy^{\e})=\e^{\beta}|Y_\e(t_\e)\right)|\ \tauy_\e>t_\e\right] \notag
\\
=&\E\left[\Pp\left(Y_\eps(\tauy^{\e})=\e^{\beta}|Y_\e(t_\e)\right);|Y_\e(t_\e)|<\e^{\beta''}\, \big|\, \tauy_\e>t_\e\right]+\Pp\left(Y_\e(t_\e)>\e^{\beta''}|\ \tauy_\e>t_\e\right) + o(1)\notag
\\=&A_1+A_2+o(1).
\label{eq:strong-Markov-formula-decomp}
\end{align}
Due to~\eqref{eq:deviation-prob}, 
\begin{align}\label{eq:small-Y-conditioned}
A_1&=\Pp\left(|Y_\e(t_\e)|<\e^{\beta''};\sup_{0\leq t\leq t_\e}|\Delta_\e(t)|\leq\e^{\beta''}|\ \tauy_\e>t_\e\right)+o(1)
\notag\\&\leq \Pp\left(|Z_\e(t_\e)|<2\e^{\beta''}|\ \tauy_\e>t_\e\right)+o(1)\to 0.
\end{align}
In the last convergence, we used \eqref{eq:Gaussian-density-asymptotics} to compute
\[\Pp(|Z_\e(t_\e)|<2\e^{\beta''})=2\eps^{\beta''}\eps^{1-\alpha+\beta}e^{\lambda C}\psi_0(y)(1+o(1)),\]
and we used \eqref{eq:stopping-time-sandwich} with $B_\eps\equiv \Omega$, along with Lemma~\ref{lem:main-beta-theorem-for-linear} to compute
$\Pp(\tauy_\eps>t_\eps)=2 e^{\lambda C}\psi_0(y)\eps^{\alpha-1}(1+o(1))$, so $\Pp(|Z_\e(t_\e)|<2\e^{\beta''})/\Pp(\tauy_\eps>t_\eps)\to 0$  follows from our assumptions on $\alpha,\beta,\beta''$.

Also due to~\eqref{eq:deviation-prob},
\[
\Pp\left(Z_\e(t_\e)>2\e^{\beta''}|\tauy_\e>t_\e\right)+o(1)\leq\ A_2\ \leq \Pp\left(Z_\e(t_\e)>0|\tauy_\e>t_\e\right)+o(1).
\]

Using \eqref{eq:stopping-time-sandwich} with $B_\eps=\{Z_\e(t_\e)>2\e^{\beta''}\}$, $B_\eps=\{Z_\e(t_\e)>0\}$, and $B_\eps=\Omega$, we can switch conditioning to that in terms of $Z_\eps$:
\[
\Pp\left(Z_\e(t_\e)>2\e^{\beta''}|\eta_\e^Z>t_\e\right)+o(1)\leq\ A_2\ \leq \Pp\left(Z_\e(t_\e)>0|\tau_\e^Z>t_\e\right)+o(1),
\]
where both the left and the right hand side converge to $1/2$ as $\e\downarrow 0$ due to Lemma \ref{lem:X-equid}.
Combining this with \eqref{eq:strong-Markov-formula-decomp}, \eqref{eq:small-Y-conditioned}, noticing that all the $o(1)$ terms in these estimates are independent of the starting point $y$, and using
\eqref{eq:exit-eps-beta-Y-small-alpha-no-direction}
we obtain~\eqref{eq:exit-eps-beta-Y-small-alpha}, which completes the proof.
\end{proof}

\section{Extension to arbitrary timescales}\label{sec:extension}
The goal of this section is to extend Theorem \ref{thm:Y-exits-with-X} for arbitrary $\alpha>1$ and thus prove Theorem~\ref{thm:linear}. We set $\theta=\alpha-\beta$, $L(\e)=e^{-\frac{\lambda}{\theta} C_\e}$, and
\[
t_\e=\frac{\theta}{\lambda}\log\left(L(\e)\e^{-1}\right)=\frac{\theta}{\lambda}\log \eps^{-1}-C_\eps.
\]
 When $\theta\in(1-\beta,1)$, Theorem~\ref{thm:Y-exits-with-X} applies and there is nothing new to prove. Here we study the case $\theta\geq 1$.  Up to this point the only restriction on $\beta$ was $\beta\in(0,1)$. Let us now set $N=[\theta]+1\ge 2$, $\beta_0=\frac{1}{2}\left(1+\frac{\theta}{N}\right)$
 and assume 
$\beta\in (\beta_0,1)$ throughout this section.  We have 
\begin{equation}
\label{eq:choosing-number-of-intervals}
\theta<N<\frac{\theta}{1-\beta}.
\end{equation}
 We also define $t_\e'=t_\e/N$ and $t_{\e,k} = kt_\e',$ $k=0,1,\dots, N$.
Our plan is to track $Y_{\e,k}=Y_{\e}(t_{\e,k})$, $k=0,1,\dots, N$, using the results of the previous section on the short intervals $[t_{\e,k},t_{\e,k+1}]$.

 The first step is the following lemma which establishes that the process needs to stay close to the origin to delay the exit. 

\begin{lemma}\label{lem:small_channel}
We have
\begin{equation}\label{eq:small-gate}
\max_{k=0,\ldots,N-1}\Pp\left(\sup_{t\leq t_{\e,k}}|Y_\e(t)|>\e K(\e);\ \tauy_\eps>t_{\e,k+1}\right)\leq 
C_1e^{-C_2K^2(\e)}
\end{equation}
for some $C_1,C_2>0$. In particular,  
\[
\max_{k=0,\dots,N-1}\Pp\left(\max_{u=0,\dots,k}|Y_{\e,u}|>\e K(\e);\ \tauy_\eps>t_{\e,k+1}\right)\leq C_1e^{-C_2K^2(\e)}.
\]
\end{lemma}
\begin{proof} Using the strong Markov property and applying Duhamel's principle \eqref{eq:Y-Duhamel} \and \eqref{eq:boundedness_of_V_e}
to the initial condition $y$ with $|y|>\eps K(\eps)$, we reduce the lemma to 
the estimate
\begin{multline*}
\Pp\left(e^{\lambda t'_\eps}\inf_{t\le t_\eps}\left|y+\eps U_\eps(t)+\eps^2 V_\eps(t) \right|<\eps^\beta\right)
\\ \le \Pp\left(\sup_{t\le t_\eps}|U_\eps(t)|> K(\eps)-\eps^{\beta+\frac{\theta}{N}-1- \frac{\theta}{N}\log(L(\eps))}-\eps \frac{\|h\|_\infty}{2\lambda}\right),
\end{multline*}
and the desired inequality follows  by \eqref{eq:exp-martingale-ineq}
since $\beta+\frac{\theta}{N}-1>0$ due to \eqref{eq:choosing-number-of-intervals}.
\end{proof}

We now collect some results needed for our iteration scheme.
\begin{lemma}\label{lem:no-exit-long}
Let $Y_{\e}(0) = \e y$.
Then there is $c>0$ such that
\begin{equation}\label{eq:tau-result-y}
\lim_{\e\downarrow 0}\sup_{|y|\le K(\e)}e^{cy^2}\left|\e^{-\left(\frac{\theta}{N}+\beta-1\right)}\Pp\left(Y_\eps(\tauy_\e)=\pm\e^{\beta};\tauy_\e>t_\e'\right)-e^{\frac{\lambda C}{N}}\psi_0(y)\right|=0.
\end{equation}
Moreover, for any Lipschitz function $h$ on $\R$, exponentially decaying at $\infty$, we have
\begin{equation}\label{eq:h-result-for-Y}
\lim_{\e\downarrow 0}\sup_{|y|\leq K(\e)}e^{cy^2}\left|\e^{-\frac{\theta}{N}}\E\left[h(\e^{-1}Y_{\e,1});\tauy_\e>t_\e'\right]-e^{\frac{\lambda C}{N}}\psi_0(y) \int_{-\infty}^{\infty} h(y) dy\right| =0.
\end{equation}
\end{lemma}

\begin{proof}
The first claim follows from Theorem~\ref{thm:Y-exits-with-X} with $C/N$ in place of $C$ and $\alpha=\beta+\theta/N$. Note that this value of $\alpha$ belongs to $(1,1+\beta)$ due to~\eqref{eq:choosing-number-of-intervals}.

The second claim is a direct consequence of the first one and 
\begin{equation}\label{eq:Yh}
\lim_{\e\downarrow 0}\sup_{|y|\leq K(\e)}\left|\e^{-(1-\beta)}\E\left[h(\e^{-1}Y_{\e,1})|\, \tauy_\e>t_\e'\right]- \frac{1}{2}\int_{-\infty}^{\infty} h(y)dy\right|=0.
\end{equation}
Once again, to prove this, we would like to use the result for the linear process. However, the estimate \eqref{eq:deviation-prob} is insufficient when applied directly. 
Instead, let us note that 
\begin{equation*}
\e^{-1}Y_{\e,1}-\e^{-1}Z_{\e,1}=e^{\lambda t_\e'}\int_0^{t_\e'}e^{-\lambda t}\left(\sigma(Y_\e(t))-\sigma_0\right)dW(t)=\e^{-\theta/N}I_{\e}(t_\e'),
\end{equation*}
where $I_{\e}(t_\e')=e^{-\lambda C_\eps/N}\int_0^{t_\e'}e^{-\lambda t}\left(\sigma(Y_\e(t))-\sigma_0\right)dW(t)$ and choose any 
\[
\beta'\in\left(\frac{1}{2}\left(1+\frac{\theta}{2}\right),\beta\right).
\]
The exponential martingale inequality (Lemma~\ref{lem:exp-marting-ineq})  and the Lipschitz continuity of $\sigma$ imply
\[
\Pp\left(\sup_{s\le \ttau_\eps}|I_\eps|>\eps^{\beta'}\right)= o_{\exp}(1),
\]
and thus
\[
\Pp\left(\e^{-(1-\beta)}\left|\e^{-1}Y_{\e,1}-\e^{-1}X_{\e,1}\right|>\e^{\beta'-\theta/N-1 +\beta}|\, \tauy_\e>t_\e'\right)= o_{\exp}(1).
\]
Note that $\beta'-\theta/N-1+\beta>0$ due to the choice of $\beta$ and $\beta'$.
Now we can use the last display and the Lipschitz continuity of $h$ to obtain
\[
\e^{-(1-\beta)}\E\left[\left|h(\e^{-1}Y_{\e,1})-h(\e^{-1}Z_{\e,1})\right||\tauy_\e>t_\e'\right]\leq \|h\|_{Lip}\e^{\beta'-\theta/N-1 +\beta} + o_{\exp}(1)\to 0
\]
as $\e\downarrow 0$. This and \eqref{eq:equidistribution-test-function} imply~\eqref{eq:Yh}, which completes the proof of the lemma.
\end{proof}

Finally, the next Theorem implies Theorem \ref{thm:linear} with $\beta_0=\frac{1}{2}\left(1+\frac{\theta}{N}\right)$.

\begin{theorem}
There is $c>0$ such that
\begin{equation}\label{eq:last-result-time}
\lim_{\e\downarrow 0}\sup_{|y|\leq K(\e)}  e^{cy^2} \left|\e^{-(\alpha-1)}\Pp\left(\tauy_\e>t_\eps\right)-2e^{\lambda C}\psi_0(y)\right|=0,
\end{equation}
and
\begin{equation}\label{eq:last-result-direction}
\lim_{\e\downarrow 0}\sup_{|y|\leq K(\e)}\left|\Pp\left(Y_\e\left(\tauy_\e\right)=\pm\e^{\beta}\bigg|\ \tauy_\e>t_\e\right)-\frac{1}{2}\right|=0.
\end{equation}
\end{theorem}

\begin{proof} We will use $o_K(1)$ to denote any function that decays faster than any power of $\e$ as $\e\downarrow 0$.  It suffices to prove the theorem in the case where the function $K(\eps)$ grows fast enough as $\eps\to0$  to guarantee that the right-hand side of \eqref{eq:small-gate} is $o_K(1)$. To see that  the theorem will then follow in full generality, we just notice that 
 enlarging the set of initial conditions  $y$ from   $\{|y|\le K(\eps)\}$ to $\{|y|\le K(\eps)\vee |\log \eps|\}$ reduces the situation to that special case.

We will prove by induction that for every $k=1,\dots, N,$ there is $c>0$ such that
\begin{equation}\label{eq:ind_hyp}
\lim_{\e\downarrow 0}\sup_{|y|\leq K(\e)}e^{cy^2}\left|\e^{-\left(k\frac{\theta}{N}+\beta-1\right)}\Pp_y\left(\tauy_\e>t_{\eps,k}\right)-2e^{\frac{\lambda k C}{N}}\psi_0(y)\right|=0,
\end{equation}
where we explicitly indicate the dependence on the starting point $Y_\e(0)=\e y$  as a subscript 
in~$\Pp_y$ for clarity. The case $k=N$ is the desired result~\eqref{eq:last-result-time}. The base of induction, the case $k=1$, is the first claim of Lemma \ref{lem:no-exit-long}.   Let us make the
induction step assuming that~\eqref{eq:ind_hyp} holds for some $k$.

Lemma \ref{lem:small_channel} and the Markov Property allows us to write
\begin{multline}\label{eq:ind-step-small-gate}
\Pp_{y}\left(\tauy_\e>t_{\e,k+1}\right) = \Pp_{y}\left(\tauy_\e>t_{\e,k+1}; |Y_{\e,1}|\leq\e K(\e)\right) + o_K(1)\\
=\int_{-K(\e)}^{K(\e)}\Pp_{y'}\left(\tauy_\e>t_{\e,k}\right)\Pp_{y}\left(Y_{\e,1}\in\e dy';\tauy_\e>t_\e'\right)+o_K(1)
\end{multline}
Using the induction hypothesis \eqref{eq:ind_hyp}, we obtain
\[
\Pp_{y'}\left(\tauy_\e>t_{\e,k}\right)=2\e^{k\frac{\theta}{N}+\beta-1}e^{\frac{\lambda k C}{N}}\psi_0(y')+e^{-cy'^2}o\left(\e^{k\frac{\theta}{N}+\beta-1}\right),
\]
where the error term is uniform over $|y'|\leq K(\e)$. This means that the first term on the right hand side of \eqref{eq:ind-step-small-gate} can be written as
\begin{align}\label{eq:integrated-Markov}
\int_{-K(\e)}^{K(\e)}&\Pp_{y}\left(\tauy_\e>t_{\e,k+1}|Y_{\e,1}=\e y'\right)\Pp_{y}\left(Y_{\e,1}\in\e dy';\tauy_\e>t_\e'\right)\\
\notag&=2 \e^{k\frac{\theta}{N}+\beta-1}e^{\frac{\lambda k C}{N}}\E_{y}\left[\psi_0\left(\e^{-1}Y_{\e,1}\right);|Y_{\e,1}|\leq\e K(\e);\tauy_\e>t_\e'\right] +H(\e,y),
\end{align}
where the error term satisfies
\begin{equation}\label{eq:error}
\sup_{|y|\leq K(\e)}|H(\e,y)|\\
\leq\sup_{|y|\leq K(\e)}\E_{y}\left( e^{-c(\e^{-1}Y_{\e,1})^2} ;\tauy_\e>t_\e'\right)\cdot o \left(\e^{k\frac{\theta}{N}+\beta-1}\right)=o \left(\e^{(k+1)\frac{\theta}{N}+\beta-1}\right)
\end{equation}
where we used \eqref{eq:h-result-for-Y} in the last step with $h(y)=e^{-cy^2}$. 
The main term on the right-hand side of~\eqref{eq:integrated-Markov} can be estimated using Lemma~\ref{lem:small_channel} and \eqref{eq:h-result-for-Y} with $h(y)=\psi_0(y)$ (so  $\int_\R h(y)dy=1$):
\begin{multline}\label{eq:exp-asym}
\E_{y}\left[\psi_0\left(\e^{-1}Y_{\e,1}\right);|Y_{\e,1}|\leq\e K(\e);\tauy_\e>t_\e'\right] =\E_{y}\left[\psi_0\left(\e^{-1}Y_{\e,1}\right);\tauy_\e>t_\e'\right] +o_K(1)\\
=\e^{\frac{\theta}{N}}e^{\frac{\lambda C}{N}}\psi_0(y) + o\left(\e^{\frac{\theta}{N}}\right),
\end{multline}
and this expansion holds uniformly over $|y|\leq K(\e)$.

Putting together \eqref{eq:ind-step-small-gate}, \eqref{eq:integrated-Markov}, \eqref{eq:error}, and \eqref{eq:exp-asym} yields, for sufficiently small $c'>0$,
\begin{equation*}
\lim_{\e\downarrow 0}\sup_{|y|\leq K(\e)}e^{c'y^2}\left|\e^{-\left((k+1)\frac{\theta}{N}+\beta-1\right)}\Pp_y\left(\tauy_\e>t_{\eps,k+1}\right)-2e^{\frac{\lambda (k+1)C}{N}}\psi_0(y)\right|=0.
\end{equation*}
 This completes the induction step and finishes the proof of~\eqref{eq:last-result-time}. 

 To prove \eqref{eq:last-result-direction}, note first that \eqref{eq:small-gate} and the strong Markov property implies
\begin{multline*}
\Pp\left(Y_\e\left(\tauy_\e\right)=
\pm\e^{\beta}\bigg|\tauy_\e>t_\e\right)=\\
\int_{-K(\e)}^{K(\e)} \Pp_y\left(Y_\e\left(\tauy_\e\right)=\pm\e^{\beta}\bigg|\tauy_\e>t_{\e,1}\right)\Pp\left(Y_{\e,N-1}\in \e dy|\tauy_\e>t_{\e,N-1}\right)+o(1).
\end{multline*}
Using \eqref{eq:tau-result-y}, the integrand can be written as
\[
\Pp_y\left(Y_\e\left(\tauy_\e\right)=\pm\e^{\beta}\bigg|\tauy_\e>t_{\e,1}\right)=\frac{\e^{\frac{\theta}{N}+\beta-1}e^{\lambda C}\psi_0(y)+o\left(\e^{\frac{\theta}{N}+\beta-1}\right)}{2\e^{\frac{\theta}{N}+\beta-1}e^{\lambda C}\psi_0(y)+o\left(\e^{\frac{\theta}{N}+\beta-1}\right)}=\frac{1}{2}+o(1),
\]
where the error terms are uniform in $y$ and thus another application of \eqref{eq:small-gate} finishes the proof of \eqref{eq:last-result-direction}.
\end{proof}

\section{Rare transitions in heteroclinic networks}
\label{sec:heteroclinic}

In this section, we discuss, briefly and nonrigorously, the questions that lead us to study the tails of exit times in detail. These questions originate in the long-term behavior of diffusions near heteroclinic networks in the vanishing noise limit. 
A heteroclinic network is a feature of the phase portrait associated with a vector field composed  
of multiple hyperbolic critical points (``saddles'') connected to each other by heteroclinic orbits, see an example of a phase portrait with a heteroclinic network for a cellular flow on Figure~\ref{fig:cellular}.

\begin{figure}
\begin{center}
\includegraphics[width=8cm]{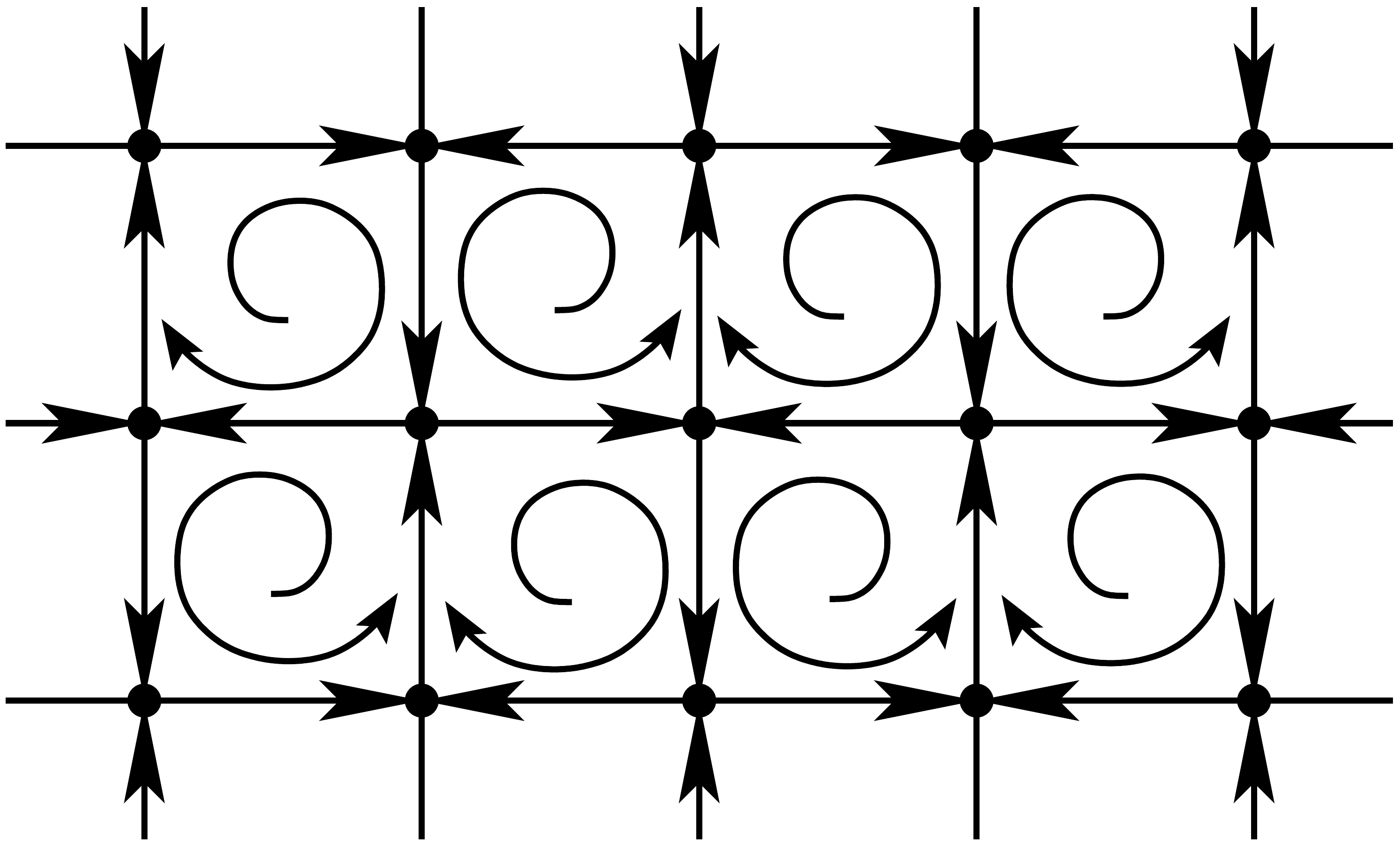}
\caption{\small A heteroclinic network is the backbone of this phase portrait associated with a cellular flow.}
\end{center}
\label{fig:cellular}
\end{figure}

Let us consider a diffusion process solving the It\^o equation
\eqref{eq:SDE} in $\R^2$ (although with minor modifications the discussion below applies to higher dimensions as well) with drift $b$ giving rise to a heteroclinic network.
The typical behavior of such processes for times that are of the order of $\log \eps^{-1}$ was studied in~\cite{Bak2011}, \cite{Bak2010}, \cite{AB2011}. 
Its main features depend mainly on the linearization of the drift $b$ near the saddle points and can be described as follows. Upon reaching a small neighborhood of a saddle, the  process spends a long (logarithmic in $\eps$) time in that neighborhood where the drift is weak and eventually exits along the unstable manifold associated with the positive eigenvalue $\lambda$ of the linearization. This manifold is composed of two outgoing heteroclinic orbits, so the dynamics chooses one of them and follows it for a time of the order of $1$ until it reaches the next saddle where it will also eventually decide between two outgoing directions, etc. 
This description may seem to imply the picture where the limiting (as $\eps\to 0$) process is essentially a random walk on the directed graph of heteroclinic connections.
However, the character of the limiting  process is often not Markovian and depends on the linearizations of $b$ near the saddles.

To see what is going on, let us consider a 2-dimensional diffusion $(X_\eps,Y_\eps)$ near a model saddle 
described by equations
\begin{align*}
dX_{\eps}(t)&=\lambda X_{\eps}(t)dt+\eps dW(t),\\
dY_{\eps}(t)&=-\mu Y_{\eps}dt+ \eps dB(t),
\end{align*}
driven by independent standard Wiener processes $W$ and $B$. Here $\lambda>0$ and $\mu>0$ can be viewed as coefficients of expansion and contraction, respectively. Assuming that $(X_\eps(0),Y_\eps(0))=(0,1)$, i.e., starting the process on the stable manifold of the saddle located at the origin, we are interested in the distribution of $(X_\eps(\tau_\eps),Y_\eps(\tau_\eps))$, where $\tau_\eps=\inf\{t\ge 0: |X_\eps(t)|=1\}$ is the exit time from the strip $[-1,1]\times\R$. Since 
\begin{align*}
X_{\eps}(t)&=\eps e^{\lambda t} N(t),\\
Y_\eps(t)&=e^{-\mu t}+\eps M(t),
\end{align*}
where
\[N(t)=\int_0^t e^{-\lambda s} dW(s)\stackrel{d}{\longrightarrow} N(\infty),\qquad M(t)=\int_0^t e^{-\mu(t-s)} dB(s)\stackrel{d}{\longrightarrow} M(\infty),\quad \text{\rm\ as\ } t\to\infty,\]
we find that for small $\eps$,
\begin{equation}
\label{eq:exit-time-asymptot}
\tau_\eps \stackrel{d}{\approx}\frac{1}{\lambda}\log \frac{1}{\eps}+\frac{1}{\lambda}\log\frac{1}{|N(\infty)|},
\end{equation}
and
\begin{equation}
\label{eq:exit-location}
Y_\eps(\tau_\eps) \stackrel{d}{\approx} \eps^{\rho} |N(\infty)|^\rho+\eps M(\infty),
\end{equation}
where $\rho=\mu/\lambda$. This allows us to conclude that as $\eps\to 0$, the distribution of the exit point~$Y_\eps(\tau_\eps)$ depends
cruciially on how $\rho$ compares to $1$. In particular, if $\rho<1$ (i.e., the contraction is not as strong as expansion: $\mu<\lambda$
) then the first term $\eps^{\rho} |N(\infty)|^\rho$
dominates. It is positive and scales as $\eps^\rho\gg \eps$, i.e., it is stronger than the noise magnitude $\eps$. Therefore, in this situation, at the next saddle point the system is most likely to stay on the same side of the heteroclinic network.  If $\rho> 1$, then the probability of choosing either of the two outgoing connections at the next saddle approaches $1/2$ as $\eps\to0$.

This analysis can be extended to more general initial conditions, to nonlinear drift and diffusion coefficients, to higher dimensions, and to sequences of saddles. The result is that for each sequence of saddles one can iteratively determine the asymptotic probability of realization of each next step along that sequence and the scaling asymptotics of the associated exit distributions. This was done in~\cite{Bak2011}, \cite{Bak2010}, and~\cite{AB2011}.
The result is that at the logarithmic time scales (the saddle exit times are typically logarithmic in~$\eps$, see \eqref{eq:exit-time-asymptot}), certain pathways in the network are typical but many pathways are not realized due to the largely one-sided exit distributions scaling
as $\eps^\alpha$ with $\alpha<1$ which in turn are due to insufficient contraction at saddles.

This is interesting per se and among other applications gives an explanation of the poor vocabulary of excitation patterns in neural networks and similar dynamics modeled by Lotka--Volterra type systems with small noise.  However, this information is not sufficient to address questions about the time scales that are longer than logarithmic such as the limiting behavior of the invariant distribution.
To answer these questions, one must quantify the probabilities of rare events corresponding to atypical exits from saddle points. This means that one needs to study probabilities of events like 
$\Pp(Y_\eps(\tau_\eps)\sim \eps)$ for $Y_\eps(\tau_\eps)$ given in~\eqref{eq:exit-location}. The second term in \eqref{eq:exit-location} is of the order of~$\eps$, so ignoring many technical details we reduce this question to estimating 

\begin{equation}
\label{eq:poly-asymp-wrong-exit}
\Pp(\eps^\rho|N(\infty)|^\rho\sim\eps)\sim\Pp(N(\infty)\sim\eps^{\gamma}) \sim c\eps^{\gamma}
\end{equation}
with $\gamma=\frac{1}{\rho}-1$. The last relation holds 
since $N(\infty)$ has continuous Lebesgue density at $0$.

This means that the probability of an atypical exit from the saddle is asymptotically polynomial in $\eps$, of the order of $\eps^\gamma$ which in turn means that one typically has to wait for time of the order of $\eps^{-\gamma}\log \eps^{-1}$ before one sees such a rare event happen.  Ordering all exponents emerging in such calculations for all rare transitions:
$\gamma_1<\gamma_2<\ldots<\gamma_N$ and introducing $T_{k,\eps}=\eps^{-\gamma_k}\log \eps^{-1}$, we see that for  $t_\eps$ satisfying
\[T_{k,\eps}\ll t_\eps \ll T_{k+1,\eps}, \]
transitions can be classified into admissible (that typically occur many times up to $t_\eps$) and rare (that typically do not occur at all up to $t_\eps$).  As  $t_\eps$ crosses a level $T_{k,\eps}$  from below, some new transitions become available. Increasing $t_\eps$ gradually from $0$ to values beyond $T_{N,\eps}$ creates a hierarchical structure of merging clusters and associated time scales such that at each time scale, the system explores one cluster making no transitions between different clusters. 

This picture containing the description of the limit of invariant distribution, homogenization results, etc., is similar to the Freidlin--Wentzell picture of metastability and the associated hierarchy of cycles. The important difference is that in our picture the probabilities of rare events decay polynomially and the associated time scales grow polynomially while the large deviation estimates in the Freidlin--Wentzell theory lead to exponentialy decaying probabilitites and exponentially growing transition times between the metastable states. 

We do not have a rigorous derivation of a general precise version of the asymptotic relation~\eqref{eq:poly-asymp-wrong-exit}. The difficulties that emerge are related to handling nonlinearities  in the drift and diffusion terms and to the fact that $N(\infty)$ and $M(\infty)$ are only approximations to the (mutually dependent) random variables $N_\eps(\tau_\eps)$ and $M_\eps(\tau_\eps)$, where the $\eps$ subscript of $N_\eps$ and $M_\eps$ refers to the fact that for the case of non-additive noise, these stochastic processes do depend on $\eps$.  So, to realize this program, among other things we must either prove that the density of $N_\eps(\tau_\eps)$ uniformly converges to the Gaussian density in a small neighborhood of zero, or to find other means to compare distribution of   $N_\eps(\tau_\eps)$ to the Gaussian at small scales, which requires going beyond the known weak convergence of distributions.

According to~\eqref{eq:exit-time-asymptot}, if $N(\infty)$ takes an atypically small value of the order of $\eps^{\gamma}$, then
\begin{equation}
\label{eq:long-tau}
\tau_\eps\approx\frac{1+\gamma}{\lambda}\log\frac{1}{\eps},
\end{equation}
i.e., exit takes abnormally long time (the typical exit time corresponds to $\gamma=0$). In other words, the rare transitions determining the long-term behavior of diffusions near heteroclinic networks occur due to atypically long stays in the neighborhood of saddle points withstanding the repulsion in the unstable direction.

In the present paper (as well as in~\cite{BPG2017}), we study the polynomial decay of the distribution of exit times at scales described by~\eqref{eq:long-tau}. We believe that the method we propose here is applicable in the multi-dimensional situation and we plan to give a rigorous treatment of it  in upcoming publications.

\bibliographystyle{alpha}
\bibliography{citations}

\begin{thebibliography}{BPG19}

\bibitem[AB11]{AB2011}
S.~Almada and Y.~Bakhtin.
\newblock Normal forms approach to diffusion near hyperbolic equilibria.
\newblock {\em Nonlinearity}, 24:1883--1907, 2011.

\bibitem[Bak10]{Bak2010}
Yuri Bakhtin.
\newblock Small noise limit for diffusions near heteroclinic networks.
\newblock {\em Dynamical Systems}, 25(3):413--431, 2010.

\bibitem[Bak11]{Bak2011}
Y.~Bakhtin.
\newblock Noisy heteroclinic networks.
\newblock {\em Probab. Theory Relat. Fields}, 150:1--42, 2011.

\bibitem[Bas11]{Bass:MR2856623}
Richard~F. Bass.
\newblock {\em Stochastic processes}, volume~33 of {\em Cambridge Series in
  Statistical and Probabilistic Mathematics}.
\newblock Cambridge University Press, Cambridge, 2011.

\bibitem[BPG19]{BPG2017}
Yuri Bakhtin and Zsolt Pajor-Gyulai.
\newblock Malliavin calculus approach to long exit times from an unstable
  equilibrium.
\newblock {\em Ann. Appl. Probab.}, 29(2):827--850, 04 2019.

\bibitem[CV16]{Champagnat2016}
Nicolas Champagnat and Denis Villemonais.
\newblock Exponential convergence to quasi-stationary distribution and
  ${Q}$-process.
\newblock {\em Probability Theory and Related Fields}, 164(1):243--283, Feb
  2016.

\bibitem[Eiz84]{Eizenberg:MR749377}
Alexander Eizenberg.
\newblock The exit distributions for small random perturbations of dynamical
  systems with a repulsive type stationary point.
\newblock {\em Stochastics}, 12(3-4):251--275, 1984.

\bibitem[KS91]{KS1991}
I.~Karatzas and S.E. Shreve.
\newblock {\em Brownian Motion and Stochastic Calculus}.
\newblock Graduate Texts in Mathematics. Springer New York, 1991.

\end{thebibliography}

\end{document}